\documentclass[11pt]{article}      
\usepackage[margin=1in]{geometry}  

\usepackage{libertine} 
\usepackage{inconsolata} 

\usepackage{microtype}
\usepackage[defaultlines=3,all]{nowidow}
\usepackage[export]{adjustbox}
\usepackage{xspace}
\usepackage{amsmath}
\usepackage{amsthm}
\usepackage{wrapfig}
\usepackage{thm-restate}
\usepackage{tikz}
\usepackage{textcomp}

\usepackage{tcolorbox}

\usepackage[font=small]{caption}
\usepackage[hidelinks]{hyperref}
\usepackage[capitalise,nameinlink]{cleveref}
\usepackage{amsfonts}
\usepackage{amssymb}
\usepackage[textsize=tiny,disable]{todonotes}
\usepackage[inline]{enumitem}
\usepackage{mathtools}
\usepackage{cases}
\usepackage[only,llbracket,rrbracket]{stmaryrd}

\usepackage{array}
\usepackage{multirow}
\usepackage{multicol}
\usepackage{mathrsfs} 
\newtheorem{theorem}{Theorem}[section]

\newtheorem{lemma}[theorem]{Lemma}
\newtheorem{claim}[theorem]{Claim}

\newtheorem{conjecture}[theorem]{Conjecture}
\newtheorem{question}[theorem]{Question}
\newtheorem{property}[theorem]{Property}
\theoremstyle{definition}
\newtheorem{definition}[theorem]{Definition}
\newtheorem{observation}[theorem]{Observation}

\newtheorem{remark}[theorem]{Remark}

\usetikzlibrary{arrows.meta}
\usetikzlibrary{shapes.geometric}
\usetikzlibrary{decorations.pathreplacing}

\usepackage{enumitem}
\setlist[itemize]{topsep=4pt,itemsep=3pt,parsep=0pt} 
\setlist[enumerate]{topsep=4pt,itemsep=3pt,parsep=0pt} 

\Crefname{figure}{Figure}{Figures}

\renewcommand{\int}{\mathrm{int}}

\newcommand{\rk}{\mathrm{rk}}

\newcommand{\N}[0]{\mathrm{\mathbb{N}}}

\renewcommand{\phi}{\varphi}

\newcommand{\Start}{\mathrm{start}}
\newcommand{\End}{\mathrm{end}}

\newcommand{\gc}[1]{\mathcal{#1}}

\newcommand{\CC}{\mathscr{C}}

\renewcommand{\le}{\leqslant}
\renewcommand{\leq}{\le}
\renewcommand{\ge}{\geqslant}
\renewcommand{\geq}{\ge}
\newcommand{\KK}{\mathcal{K}}

\newenvironment{claimproof}[1][\proofname]{%
  \begin{proof}[#1]%
}{%
  \end{proof}%
}

\newcommand{\from}{\colon}

\begin{document}

\newcommand{\funding}{JD and SzT received funding from the European Research Council (ERC) with grant agreement No.\ 101126229 -- {\sc BUKA}.\\
NM received funding from the European Union through an ERA Fellowship with grant agreement No.\ 101334340 -- {\sc LoCoMoDe}.\\[0.5em]
\hspace*{\fill}\includegraphics[width=40px]{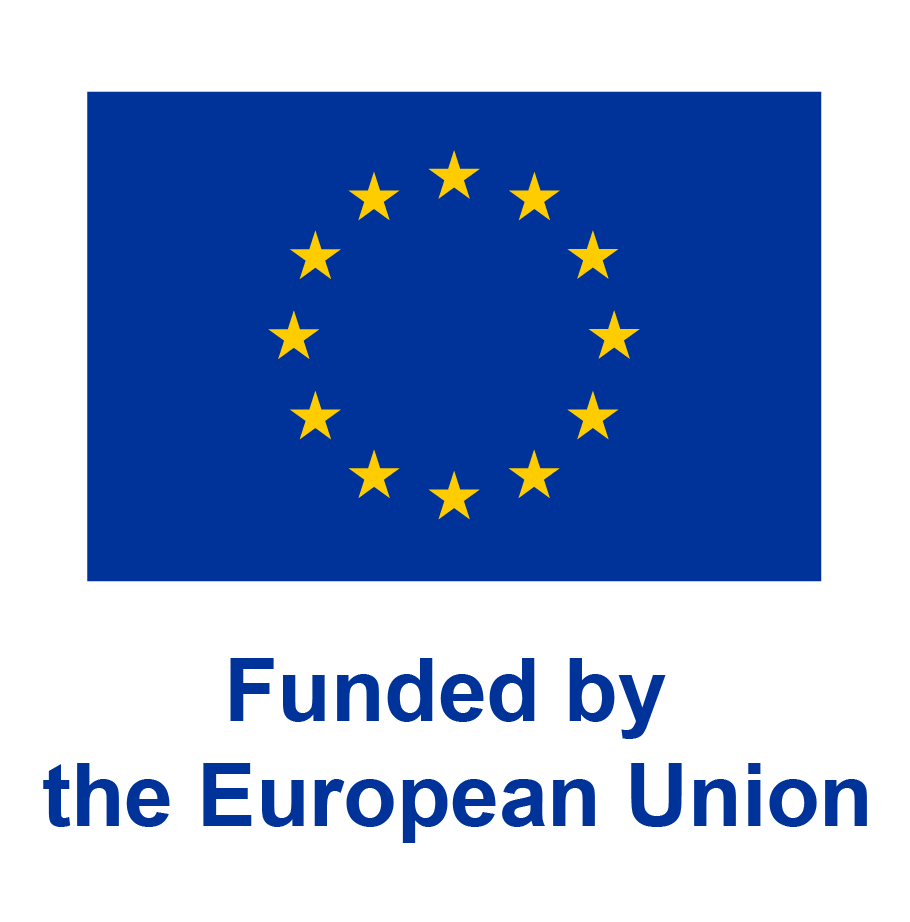}\hspace{1em}\includegraphics[width=40px]{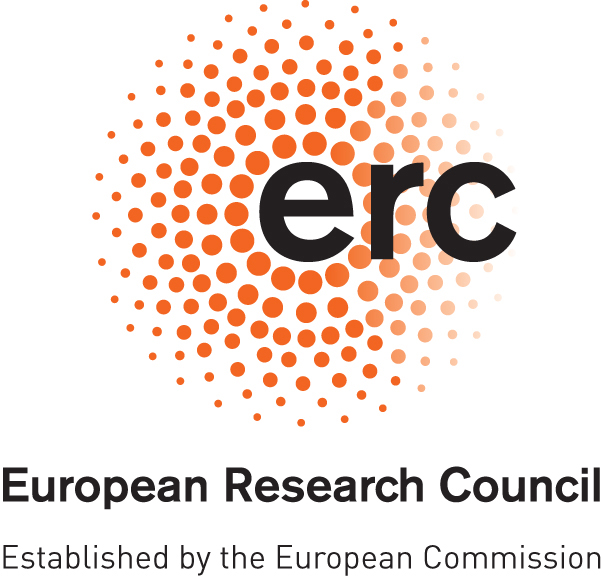}
}

\title{Hereditary $2$-WQO Graph Classes Have Bounded Clique-Width\footnote{\funding}
}
\date{\today}
\author{
  Julien Duron \\
  \small{University of Warsaw} \\
  \small{\texttt{j.duron@uw.edu.pl}}
  \and
  Nikolas M\"ahlmann \\
  \small{University of Warsaw} \\
  \small{\texttt{maehlmann@mimuw.edu.pl}}
  \and
  Szymon Toruńczyk\\
  \small{University of Warsaw} \\
  \small{\texttt{szymtor@mimuw.edu.pl}}
}
\maketitle
\begin{abstract}
A graph class is $k$-WQO if its $k$-labeled graphs are well-quasi-ordered under label-preserving induced subgraph embeddings.
We show that every hereditary graph class that is $2$-WQO has bounded clique-width.
Combined with the recent result of Dumas and Lopez, this confirms a long-standing conjecture of Pouzet:
A hereditary graph class is $2$-WQO if and only if it is $k$-WQO for all $k\geq 2$, if and only if it is $\forall$-WQO, that is, its labeled graphs are well-quasi-ordered for  every possible well-quasi-ordered label set.

Our proof builds on a recent structure/non-structure dichotomy for the model theoretic notion of monadic dependence by  Dreier, Mählmann, and Toruńczyk.
Through the non-structure characterization by forbidden induced subgraphs, we show that every hereditary $2$-WQO graph class is monadically dependent.
Leveraging the Ramsey-theoretic structural properties provided by monadic dependence, we then establish bounded clique-width by ruling out the existence of large well-linked sets, which are the canonical obstructions for clique-width.  
\end{abstract}

\noindent
\textbf{Acknowledgements.}
We thank Pierre Simon for suggesting the connection of Pouzet's conjecture to monadic dependence.

\section{Introduction}
A graph class $\CC$ is \emph{well-quasi-ordered} (WQO) if it contains no infinite antichain under the induced subgraph (embedding) relation. Further, if $k$ is a positive integer,
we say  that $\CC$ is \emph{$k$-WQO} if 
the class of all vertex-labelled graphs from~$\CC$ with labels in $[k]$ is well-quasi-ordered under label-preserving embeddings. 

For example, the class of cycles is not WQO, while the class of paths is WQO, but it is not $2$-WQO, since labelling the two endpoints with a distinguishing label creates an infinite antichain. 
On the other hand,
the class of linear orders (viewed as directed graphs) is $k$-WQO for all $k\ge 1$ by Higman's lemma, and
so is the class of  cographs  \cite{Damaschke1990} (see also \cite[Thm. 1]{AtminasL15}).

The notion of $k$-WQO was studied by Pouzet around 1972, who observed that already $2$-WQO is a very strong assumption and asked
 whether for classes of relational structures which are \emph{hereditary} (that is, closed under taking induced substructures)  $2$-WQO 
implies $k$-WQO for all $k \ge 1$.
This is a rephrasing of a question posed in \cite[\textsection 2.4]{pouzet1972}, which has been repeated by
Fra\"iss\'e in \cite[Chap. 12, \textsection 3.1]{Fraisse1986TheoryRelations}. 
In the setting of graphs, this has been repeated by Daligault, Rao, and Thomass\'e~\cite{daligault2010} in the modern formulation, as follows:

\begin{conjecture}[{\cite{pouzet1972}, \cite[Conj. 1]{daligault2010}}]\label{conj:pouzet}
For hereditary graph classes, $2$-WQO implies $k$-WQO for all $k\ge 1$.
\end{conjecture}
In fact, Pouzet \cite[Problems 9 (1)]{Pouzet2024WellquasiorderingAE} conjectured that for hereditary classes, $2$-WQO even implies \emph{$\forall$-WQO}\footnote{The terminology $\forall$-WQO is from \cite{dumaslopez2026}; it is sometimes called WQO-WQO, hereditary WQO, or labelled WQO.}, which means that the class remains WQO under  embeddings which preserve or increase the labels, when the vertex labels are drawn from an arbitrary fixed well-quasi-order
(another variant of this conjecture appears in \cite[Problem~3]{POUZET2024103813}).

 Daligault, Rao, and Thomass\'e~\cite[Conj. 5]{daligault2010} proposed a route towards Pouzet's conjecture, by posing the following structural conjecture:

\begin{conjecture}[\cite{daligault2010}]
Every hereditary $2$-WQO graph class has bounded clique-width.
\end{conjecture}

Clique-width is one of the central width parameters for hereditary graph classes. It can be seen as an extension of tree-width to dense graphs. Roughly, graphs of bounded clique-width can be encoded in trees in a simple fashion (namely, by a fixed MSO formula). 
Thus, the conjecture of Daligault, Rao, and Thomass\'e gives an unexpected connection between a purely order-theoretic assumption and a strong structural tameness property of hereditary graph classes. This is reminiscent to the connection of minor-closed classes---which are famously well-quasi ordered by the minor relation---and tree-width discovered by Robertson and Seymour.

Thanks to the very restricted structure of graphs of bounded clique-width, and their close relationship to trees, Dumas and Lopez \cite[Thm. 4]{dumaslopez2026} were able to confirm
Pouzet's conjecture for all graph classes of bounded clique-width, using methods from automata theory.
Thus, to confirm Pouzet's conjecture for hereditary graph classes, it remains to establish the conjecture of Daligault, Rao, and Thomass\'e.

Our main result does exactly this.

\begin{theorem}\label{thm:main}
Every hereditary $2$-WQO graph class has bounded clique-width.
\end{theorem}

Combined with the result of Dumas and Lopez~\cite{dumaslopez2026}, this proves the following.

\begin{theorem}\label{thm:pouzet}
    The following conditions are equivalent for hereditary graph classes:
    \begin{center}
    \begin{enumerate*}
        \item\label{it:2wqo} $2$-WQO,
        \qquad\qquad\qquad
        \item $k$-WQO for all $k\ge 1$,
        \qquad\qquad\qquad
        \item $\forall$-WQO.
    \end{enumerate*}
    \end{center}
\end{theorem}
This confirms two conjectures of Pouzet, in the setting of graphs: \Cref{conj:pouzet} 
and the conjecture about $\forall$-WQO \cite[Problems 9 (1)]{Pouzet2024WellquasiorderingAE}.
 Also this implies that hereditary $2$-WQO graph classes have exponential \emph{speed}---contain at most $2^{O(n)}$ graphs with $n$ vertices up to isomorphism---answering a question of Pouzet and Sobrani (2003), see \cite[Problem 5]{POUZET2024103813}. This follows from \Cref{thm:main}, since classes of bounded clique-width are well-known to have at most exponential speed~\cite{Allen_Lozin_Rao_2009}.

 Note that by \cite{LOZIN20181},
 there are hereditary WQO graph classes with unbounded clique-width, so the assumption of $2$-WQO is indeed necessary for our result\footnote{We don't know whether the heredity assumption  can be dropped in this statement. If \Cref{it:2wqo} is replaced by $3$-WQO then heredity can be dropped, since removing vertices can be simulated by one extra color.}.\todo{Reviewer A: What about non-hereditary classes that are $2$-well-quasi-ordered? Is
  there something to be said?}

\medskip
To prove Theorem~\ref{thm:main}, we 
work with 
a family of forbidden induced configurations originally introduced by Brignall and Cocks \cite{Brignall-Cocks-Framework}, which we call \emph{Brignall-Cocks patterns}.
These are a generalization of the path-like structures that were recently used by Mählmann~\cite{mahlmann2025forbidden} to characterize graph classes of unbounded shrub-depth.
Every hereditary $2$-WQO class is \emph{pattern-free}:
it cannot contain such patterns of arbitrarily large size, since otherwise,  
 a labelled antichain can be exhibited easily. This is the first step of the proof; see \Cref{fig:overview} for a diagrammatic overview of the~proof.

As a next step, we connect pattern-freeness to a general model-theoretic notion of tameness, namely \emph{monadic dependence}, introduced by Shelah \cite{shelah1986monadic}. Monadically dependent graph classes vastly extend classes of bounded clique-width.
They have recently been characterized 
 in terms of forbidden obstructions by Dreier, M\"ahlmann, and Toru\'nczyk~\cite{dreier2024flipbreakability}.
 We observe that each of the forbidden obstructions to monadic dependence contains arbitrarily large Brignall-Cocks patterns. It follows that every pattern-free graph class is monadically dependent.

The final step is to prove that monadically dependent, $2$-WQO classes have bounded clique-width. The \emph{insulation property} developed in~\cite{dreier2024flipbreakability} provides a Ramsey-theoretic notion of structure in monadically dependent graphs.
Combining the insulation property with $2$-WQO, we are able to distribute every sufficiently large vertex set across a long linear partition whose intermediate blocks act as separators of low cut-rank. 
This lets us exclude the existence of large \emph{well-linked sets}, which are obstructions characterizing bounded rank-width by the result of Oum and Seymour~\cite{oum2006}.
Since bounded rank-width is equivalent to bounded clique-width, this completes the proof of \Cref{thm:main}.

To summarize, our proof proceeds by first observing that every hereditary 2-WQO graph class $\CC$ is monadically dependent using the characterization of the latter via forbidden obstructions, and then by leveraging the structure of monadically dependent classes, in combination with the obstructions to bounded rank-width, to prove that $\CC$ furthermore has bounded rank-width (and thus clique-width). 
Conceptually, this connects three lines of work that had so far evolved largely in parallel: labelled well-quasi-ordering, structural width measures such as rank-width and clique-width, and the recent model-theoretic structure theory of monadically dependent graph classes. In particular, it shows that a seemingly weak labelled order-theoretic assumption already forces one of the central structural tameness properties of hereditary graph classes.

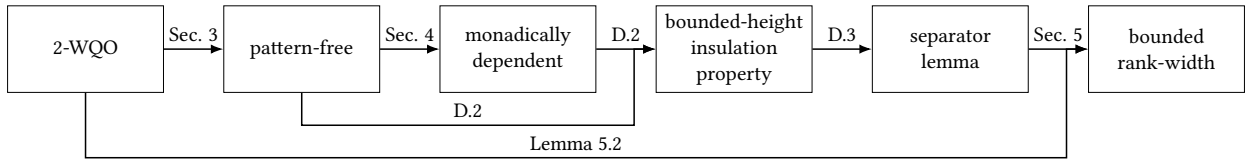
\begin{figure}[ht]
    \centering
\resizebox{\linewidth}{!}{
\begin{tikzpicture}[
  >=Latex,
  node distance=8mm and 11mm,
  box/.style={
    draw,
    align=center,
    inner sep=4pt,
    minimum height=16mm,
    text width=26mm
  }
]

\node[box] (wqo) {2-WQO};
\node[box, right=of wqo] (pat) {pattern-free};
\node[box, right=of pat] (md) {monadically\\dependent};
\node[box, right=of md] (ins) {bounded-height\\insulation property};
\node[box, right=of ins] (part) {separator\\lemma};
\node[box, right=of part] (rw) {bounded\\rank-width};

\draw[->, thick] (wqo) -- node[above] {Sec. \ref{sec:patterns}} (pat);
\draw[->, thick] (pat) -- node[above] {Sec. \ref{sec:mdep}} (md);
\draw[->, thick] (md) -- node[above] {\ref{subsec:bounded-height-insulators}} (ins);
\draw[->, thick] (ins) -- node[above] {\ref{subsec:wrapup}} (part);
\draw[->, thick] (part) -- node[above] {Sec. \ref{sec:final}} (rw);

\draw[->, thick]
  (pat.south) -- ++(0,-6mm)
  -| node[pos=0.25, above] {\ref{subsec:bounded-height-insulators}} ([xshift=-4mm]ins.west)
  -- (ins.west);

\draw[->, thick]
  (wqo.south) -- ++(0,-12mm)
  -| node[pos=0.25, above] {Lemma \ref{lem:wl-bounded}} ([xshift=-4mm]rw.west)
  -- (rw.west);

\end{tikzpicture}
}
\caption{Proof overview. The edge labels refer to corresponding sections or lemmas.}\label{fig:overview}
\end{figure}
We close this introduction by a brief discussion.
Observe that the 2-WQO property is used twice in our proof. 
We don't know whether pattern-freeness alone implies bounded rank-width. 

\begin{question}
Does every pattern-free class have bounded rank-width?
\end{question}

We also ask whether something stronger holds:
\begin{question}\label{question}
  Is every pattern-free class 2-WQO?
\end{question}
A positive answer to this question, together with \Cref{lem:2wqo-pattern-free},  would imply that for hereditary classes, 2-WQO is equivalent to pattern-freeness. However, Robert Brignall has answered \Cref{question} negatively~\cite{Brignall-example}. Thus, characterizing 2-WQO hereditary classes in terms of forbidden induced subgraphs remains an open question.

\section{Preliminaries}

\paragraph{Graphs.} We consider finite, simple, undirected graphs. 
A graph class is \emph{hereditary} if it is closed under taking induced subgraphs.
\paragraph{Well-quasi orders.}
A \emph{quasi-order} is a reflexive and transitive binary relation $\le$ on some domain.
A \emph{well quasi-order} (\emph{WQO}) is a quasi-order $\le$ such that 
 for every infinite sequence $q_1, q_2, \ldots$ of elements from the domain,
there exist $i < j$ such that $q_i \le q_j$. Equivalently, 
there are no infinite antichains and no infinite descending sequences.

For a positive integer $k$,
a graph class $\CC$ is \emph{$k$-WQO} if the class of $k$-labeled graphs from $\CC$
is WQO under label-preserving induced subgraph embeddings.
 More precisely, $\CC$ is {$k$-WQO} if for every infinite sequence
$(G_1, \ell_1), (G_2, \ell_2), \ldots$ where each $G_i \in \CC$
and $\ell_i \from V(G_i) \to [k]$ is a vertex-labeling,
there exist $i < j$ and an injective $\phi \from V(G_i) \to V(G_j)$ such that
\begin{itemize}
    \item $uv \in E(G_i) \Leftrightarrow \phi(u)\phi(v) \in E(G_j)$ for all $u, v \in V(G_i)$, and
    \item $\ell_j(\phi(v)) = \ell_i(v)$ for all $v \in V(G_i)$.
\end{itemize}


\paragraph{Flips.}
Fix a graph $G$ and a partition $\KK$ of its vertices.
For every vertex $v \in V(G)$ we denote by $\KK(v)$ the part of $\KK$ which contains $v$.
Let $F \subseteq \KK^2$ be a symmetric relation.
The \emph{flip} $G \oplus F$ of $G$ is defined as the  graph with vertex set $V(G)$,
and edges defined by the following condition, for distinct $u,v\in V(G)$:
\[
    uv \in E(G \oplus F) \Leftrightarrow \begin{cases}
        uv \notin E(G) & \text{if } (\KK(u), \KK(v)) \in F,\\
        uv \in E(G) & \text{otherwise.}
    \end{cases}
\]

We call $G \oplus F$ a \emph{$\KK$-flip} of $G$.
If $\KK$ has at most $k$ parts, we call $G \oplus F$ a \emph{$k$-flip} of $G$.

\paragraph{Rank-width.}
Instead of bounding the clique-width of hereditary $2$-WQO classes, we will bound their \emph{rank-width}, a parameter that is functionally equivalent to clique-width:

\begin{theorem}[{\cite[Prop.\ 6.3]{oum2006}}]\label{fact:rw}
    A graph class has bounded clique-width if and only if it has bounded rank-width.
\end{theorem}

The single property of rank-width we will use is its characterization by Oum and Seymour in terms of forbidden obstructions which we now describe.
The rank of two disjoint vertex sets $A,B$
of a graph $G$, denoted 
$\rk_G(A,B)$, is defined as the rank, over the two-element field, of the $(0,1)$-matrix with rows $A$ and columns $B$, where the entry at row $a\in A$ and column $b\in B$ is $1$ if $ab\in E(G)$ and $0$ otherwise.
For a single set $A \subseteq V(G)$ we define $\rk_G(A) := \rk(A,V(G) - A)$. (This is also known as the \emph{cutrank} of A.)

A set $W$ of vertices of $G$ is \emph{well-linked}, if for every bipartition $X,Y$ of $W$ and every set $Z$ with $X \subseteq Z \subseteq V(G)-Y$,
the cut-rank of $Z$ satisfies $\rk_G(Z)\ge \min(|X|,|Y|)$.

\begin{theorem}[{\cite[Thm. 5.2]{oum2006}}]\label{fact:well-linked}
Every graph of rank-width greater than $k$ contains a well-linked set of size~$k$.
\end{theorem}

We will also need the following observation.

\begin{observation}\label{obs:well-linked-hereditary}
    If $W$ is well-linked in $G$, then also every subset $W'$ of $W$ is well-linked in $G$.
\end{observation}
\begin{proof}
    Consider a bipartition $X',Y'$ of $W'$ and a set $Z$ with $X' \subseteq Z \subseteq V(G)-Y'$.
    Now for $X := W \cap Z$ and $Y := W - Z$ we have
    \[
        \rk_G(Z) \geq \min(|X|,|Y|) \geq \min(|X'|,|Y'|)    
    \]
    where the first inequality is by well-linkedness of $W$, and the second by $X \supseteq X'$ and $Y \supseteq Y'$.
\end{proof}



\section{Brignall-Cocks Patterns}
\label{sec:patterns}
We define the main obstruction to $2$-WQO that we will be using throughout the paper.
Those are patterns which were originally introduced by Brignall and Cocks \cite{Brignall-Cocks-Framework}, extending earlier work of Brignall and Cocks \cite{Brignall_Cocks_2022} and of
Collins, Foniok, Korpelainen, Lozin, and Zamaraev \cite{COLLINS2018145}. A special case of those patterns appears in the work of M\"ahlmann \cite{mahlmann2025forbidden}.
An example is depicted in \Cref{fig:pattern}.

\begin{definition}[Brignall-Cocks patterns]\label{def:pattern}
    For $m,r \in \N$,
    an \emph{$(m,r)$-pattern} is a graph $G$ with vertex set $[m]\times [r]$
      satisfying the following.
    For $j\in [r]$, denote by $L_j(G) := \{(i,j) : i \in [m]\}$ the $j$th \emph{layer}.
    \begin{enumerate}
        \item Each layer $L_j(G)$ induces either a clique or an independent set in $G$.
        \item For each pair of distinct layers $L_j(G)$ and $L_{j'}(G)$:
        \begin{itemize}
            \item If $|j- j'| > 1$, then $L_j(G)$ and $L_{j'}(G)$ are either fully adjacent or fully non-adjacent.
            \item If $|j- j'| = 1$, then there is a relation ${\sim} \in \{=,\neq,\leq,\geq,<,>\}$ such that
        \[
            (i,j)(i',j') \in E(G) \quad \Leftrightarrow \quad i \sim i'
            \quad\quad \text{for all $i,i' \in [m]$.}
        \]
        \end{itemize}
    \end{enumerate}
\end{definition}

\begin{figure}[htbp]
    \centering
\vspace{-1.8em}
\begingroup
\resizebox{0.8\linewidth}{!}{%
\begin{tikzpicture}[
  xedge/.style= {line width=0.45pt, gray!65!black, opacity=0.6},
  dot/.style   = {circle, fill=black, inner sep=2pt, outer sep=0pt},
]
\pgfmathsetmacro{\xs}{2.2}
\pgfmathsetmacro{\ys}{0.6}
\foreach \r in {1,2,3}{
  \foreach \c in {1,...,7}{
    \pgfmathsetmacro{\xp}{(\c-1)*\xs}
    \pgfmathsetmacro{\yp}{(3-\r)*\ys}
    \coordinate (u\r-\c) at (\xp, \yp);
  }
}
\foreach \r in {1,2}{ \draw[xedge] (u\r-1) to[out=90, in=90] (u\r-3); }
\draw[xedge] (u3-1) to[out=90, in=90, looseness=1.8] (u3-3);
\draw[xedge] (u1-1) -- (u2-2);
\draw[xedge] (u1-1) to[out=90, in=90] (u2-3);
\draw[xedge] (u1-1) -- (u3-2);
\draw[xedge] (u1-1) to[out=90, in=90] (u3-3);
\draw[xedge] (u2-1) to[out=90, in=90] (u1-3);
\draw[xedge] (u2-1) -- (u3-2);
\draw[xedge] (u2-1) to[out=90, in=90] (u3-3);
\draw[xedge] (u3-1) to[out=90, in=90] (u1-3);
\draw[xedge] (u3-1) to[out=90, in=90] (u2-3);
\draw[xedge] (u1-7) -- (u2-6);
\draw[xedge] (u1-7) -- (u3-6);
\draw[xedge] (u2-7) -- (u3-6);
\draw[xedge] (u1-4) -- (u2-3);
\draw[xedge] (u1-4) -- (u2-5);
\draw[xedge] (u1-4) -- (u3-3);
\draw[xedge] (u1-4) -- (u3-5);
\draw[xedge] (u2-4) -- (u3-3);
\draw[xedge] (u2-4) -- (u3-5);
\draw[xedge] (u1-1)--(u1-2)--(u1-3);
\draw[xedge] (u1-5)--(u1-6)--(u1-7);
\draw[xedge] (u2-1)--(u2-2)--(u2-3);
\draw[xedge] (u2-5)--(u2-6)--(u2-7);
\draw[xedge] (u3-1)--(u3-2)--(u3-3);
\draw[xedge] (u3-5)--(u3-6)--(u3-7);
\draw[xedge] (u1-4) -- (u2-5);
\draw[xedge] (u1-4) -- (u3-5);
\draw[xedge] (u2-4) -- (u3-5);
\draw[xedge] (u1-3) -- (u2-4);
\draw[xedge] (u1-3) -- (u3-4);
\draw[xedge] (u2-3) -- (u1-4);
\draw[xedge] (u2-3) -- (u3-4);
\draw[xedge] (u3-3) -- (u1-4);
\draw[xedge] (u3-3) -- (u2-4);
\draw[xedge] (u1-5) -- (u2-5);
\draw[xedge] (u2-5) -- (u3-5);
\draw[xedge] (u1-5) to[bend left=20] (u3-5);
\draw[xedge] (u1-7) -- (u2-7);
\draw[xedge] (u2-7) -- (u3-7);
\draw[xedge] (u1-7) to[bend left=20] (u3-7);
\foreach \c in {1,...,7}{
  \pgfmathsetmacro{\xp}{(\c-1)*\xs}
  \node[font=\small] at (\xp, -0.75) {$L_{\c}$};
}
\foreach \c/\rel in {1/{\leq}, 2/{=}, 3/{\neq}, 4/{<}, 5/{=}, 6/{\geq}}{
  \pgfmathsetmacro{\xp}{(\c-0.5)*\xs}
  \node[font=\small] at (\xp, -0.45) {$\rel$};
}
\foreach \r in {1,2,3}{
  \foreach \c in {1,...,7}{
    \node[dot] at (u\r-\c) {};
  }
}
\end{tikzpicture}%
}%
\endgroup
    \caption{A $(3,7)$-pattern. The connections between consecutive layers are governed by the indicated relations.
    Layers $L_1$ and $L_3$ are fully adjacent, while all other non-consecutive pairs are fully non-adjacent. Layers $L_5$ and $L_7$ are cliques; all other layers are independent sets.}
    \label{fig:pattern}
\end{figure}
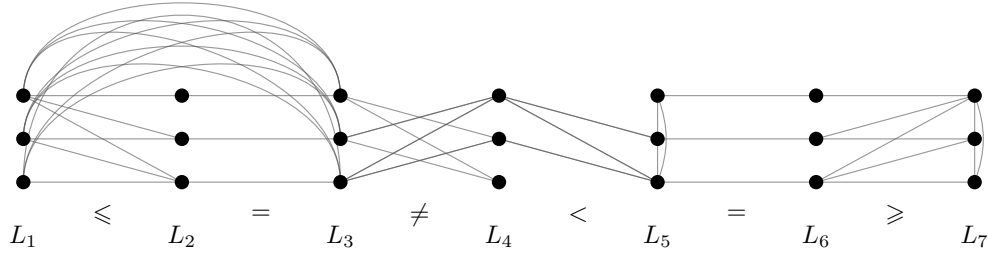

\begin{definition}\label{def:pattern-free}
    A graph is \emph{$(m,r)$-pattern-free} if it contains no $(m,r)$-pattern as an induced subgraph.
    A graph class $\CC$ is \emph{$(m,r)$-pattern-free} if every graph in the class is.
    A graph class $\CC$ is \emph{pattern-free} if there are $m,r \in \N$ such that $\CC$ is $(m,r)$-pattern-free.
\end{definition}


\begin{lemma}\label{lem:2wqo-pattern-free}
    If $\CC$ is hereditary and $2$-WQO, then $\CC$ is pattern-free.
\end{lemma}
\begin{proof}
    Assume $\CC$ is hereditary and not pattern-free. We will show that $\CC$ is not $2$-WQO.
    By assumption, $\CC$ contains an infinite sequence of graphs $G_2,G_4,G_6\ldots$ such that for each even number $r\geq 2$, $G_r$ is an $(m_r,r)$-pattern, for
    some $m_r$ which depends only on $r$ and which we will define later.
    For each $G_r$ consider the vertex-labeling $\ell_r : V(G_r) \rightarrow [2]$ that marks the first and last layer of $G_r$:
    \[
        \ell_r(v) := \begin{cases}
            1& \text{ if $v \in L_1(G_r) \cup L_{r}(G_r)$, and}\\
            2& \text{ otherwise.}
        \end{cases}
    \]
    Assume towards a contradiction that $\CC$ is $2$-WQO.
    Then there are $2 \leq r < t \in \N$ and a label preserving embedding $\phi : V(G_r) \to V(G_{t})$.
    As $\phi$ is label preserving we have 
    \[
        \phi\big(L_1(G_r)\big)  \subseteq L_1(G_{t}) \cup L_{t}(G_{t}).
    \]
    Up to reversing the order of the layers in $G_t$, which preserves the property of being an $(m_t,t)$-pattern,
    we can assume there is
    a subset $A_1 \subseteq L_1(G_r)$ of size at least $|L_1(G_r)|/2 = m_r/2$ that embeds to $\phi(A_1) \subseteq L_1(G_t)$.

    \begin{claim}\label{claim:induction}
        For each $s\in [r]$
        there is a set $A_s \subseteq L_s(G_r)$ and a layer index $w(s) \in [r]$
        such that:
        \begin{enumerate}
            \item $|A_s| \geq f^s(m_r)$, where $f(x) := (x-1)/2$, and $f^s(\cdot)$ denotes the $s$-times iteration of $f$,
            \item $\phi(A_s) \subseteq L_{w(s)}(G_t)$, and
            \item $w(1) = 1$ and $|w(s) - w(s-1)| = 1$ for all $s \geq 2$.
        \end{enumerate}
    \end{claim}

    Intuitively the claim states that we find a sequence of subsets $A_1, \ldots, A_r$ of $L_1(G_r), \ldots, L_r(G_r)$, such that the $\phi$-image of the $A_s$ is contained in a single layer $w(s)$ of $G_t$, which is neighboring the previous layer $w(s-1)$. This means the embedding can neither skip layers nor stay in the same layer.


    \begin{claimproof}[Proof of \Cref{claim:induction}]
        We proceed by induction on $s$.
        We have already constructed $A_1$ for the base case, before stating the claim.
        For the inductive step, assume we have already constructed $A_s$ and $w(s)$.
        Given a set of vertices $S$ in $G_r$ (or in $G_t$), we call a vertex $v$ \emph{mixed} to $S$, if $v \notin S$ and $v$ has a neighbor and a non-neighbor in $S$.
        Let $I_s$ be the first indices of the vertices in $A_s$, i.e., $A_s = I_s \times \{s\}$, and define $A'_{s+1} := I_s \times \{s + 1\} \subseteq L_{s+1}(G_r)$ to be the corresponding vertices from the next layer of $G_r$.
        Let 
        ${\sim} \in \{=,\neq,\leq,\geq,<,>\}$ be the relation that governs the adjacency between $A_s$ and $A'_{s+1}$ in $G_r$.
        One can easily verify that:
        \begin{itemize}
            \item If ${\sim}\in\{=,\neq\}$, then all vertices in $A'_{s+1}$ are mixed to $A_s$ in $G_r$.
            \item If ${\sim}\in\{\leq,\geq,<,>\}$, then all but one vertex in $A'_{s+1}$ are mixed to $A_s$ in $G_r$.
        \end{itemize}
        Since $\phi$ is an embedding, this means there is a subset $A''_{s+1} \subseteq A'_{s+1}$ of size at least $|A_s| - 1$ such that all vertices in $\phi(A''_{s+1})$ are mixed to $\phi(A_s)$ in $G_t$.
        Observe that:
        \begin{itemize}
            \item $\phi(A''_{s+1})$ contains no vertices from $L_{w(s)}(G_t)$. 
            
            Otherwise, mixedness to $\phi(A_s) \subseteq L_{w(s)}(G_t)$ would contradict the fact that $L_{w(s)}(G_t)$ is an independent set or a clique.
            \item $\phi(A''_{s+1})$ contains no vertices from $L_{j}(G_t)$ with $|w(s)-j|>1$. 
            
            Otherwise, mixedness to $\phi(A_s)$ would contradict the fact that $L_{w(s)}(G_t)$ and $L_{j}(G_t)$ are either fully adjacent or fully non-adjacent.
        \end{itemize}
        Hence, $\phi(A''_{s+1})$ is fully contained in $L_{w(s)+1}(G_t) \cup L_{w(s)-1}(G_t)$.
        By the pigeonhole principle, we can choose $w(s+1)\in\{w(s)+1,w(s)-1\}$ and subset $A_{s+1}$ of $A''_{s+1}$ of size at least $|A''_{s+1}|/2$ such that $\phi(A_{s+1}) \subseteq L_{w(s+1)}(G_t)$.
        This finishes the proof of the claim by induction.
    \end{claimproof}
    Let $g(x) := 2x + 1$ be the inverse of $f$ and set $m_r := g^r(1)$.
    By the above claim, $|A_r| \geq 1$ and there is a vertex $a \in A_r \subseteq L_r(G_r)$ that is labeled $\ell_r(a) = 1$ and mapped to a vertex $\phi(a) \in L_{w(r)}(G_t)$.
    Observe that for all $s\in [r]$, $w(s)$ is even if and only $s$ is. As $r$ was chosen even, we have $w(r) \neq 1$. Moreover, $w(r) \leq r$ so $w(r) \neq t > r$.
    Then $\phi(a)$ is neither part of the first nor last layer of $G_t$ and thus labeled $\ell_t(\phi(a)) = 2$.
    This contradicts the assumption that $\phi$ was label preserving.
\end{proof}

We observe that the prime example of a non-$2$-WQO class, the class of all paths, contains large patterns.

\begin{lemma}\label{lem:paths-not-pattern-free}
    The class of all paths is not pattern-free. 
    More precisely, the $t$-vertex path $P_t$ contains an $(m,r)$-pattern for $t := (r+1) \cdot m$.
\end{lemma}

\begin{proof}
    $P_t$ contains as an induced subgraph the disjoint union $mP_r$ of $m$ many paths of length $r$.
    This is an $(m,r)$-pattern, where each layer  $L_j(mP_r)$ is the union of the $j$th vertices from all the $m$ many paths.
\end{proof}

Finally, patterns are robust with regard to flips.
This has to be argued, as the flip-partition might not agree with the partition of the pattern into layers.

\begin{restatable}{lemma}{flippattern}\label{lem:flip-pattern}
    Let $\CC$ be a class that is pattern-free and $k\in \N$. Then the class $\CC^{(k)}$ of $k$-flips of graphs from $\CC$ is also pattern-free.
    More precisely, every $k$-flip of a $(k^r \cdot m, r)$-pattern contains an $(m,r)$-pattern.
\end{restatable}
\begin{proof}[Proof sketch]
    Using the pigeonhole principle, we find a still large induced subpattern on which the flip partition is a coarsening of the partition of the pattern into layers.
    It follows from the definition of a pattern that flipping between layers does not destroy the pattern. See \Cref{sec:appendix-flips} for details.
\end{proof}
\section{Pattern-free implies monadically dependent}\label{sec:mdep}

The goal of this section is to prove the following theorem.

\begin{theorem}\label{thm:2wqo-dependent}
    Every pattern-free graph class is monadically dependent.
\end{theorem}

Monadic dependence originates in model theory \cite{shelah1986monadic} (it is also called monadic NIP) and is defined through first-order transductions.
We will instead use the combinatorial characterization through forbidden induced subgraphs given by~\cite{dreier2024flipbreakability}, which we will now recall.

\smallskip

For $r \ge 1$, the \emph{star $r$-crossing} of order $n$ is the $r$-subdivision of $K_{n,n}$ (the biclique of order $n$).
More precisely, it consists of \emph{roots} $a_1,\dots,a_n$ and $b_1,\dots,b_n$
together with $r$-vertex paths $\{ \pi_{i,j} : i,j \in [n] \}$ that are pairwise vertex-disjoint (see \Cref{fig:dependent-patterns}).
We write $\pi(i,j,t)$ for the $t$th vertex of path $\pi_{i,j}$, for $t \in [r]$,
so that $\Start(\pi_{i,j}) = \pi(i,j,1)$ and $\End(\pi_{i,j}) = \pi(i,j,r)$.
We require that roots appear on no path, that each root $a_i$ is adjacent to $\{ \pi(i,j,1) : j \in [n] \}$,
and that each root $b_j$ is adjacent to $\{ \pi(i,j,r) : i \in [n] \}$.
The \emph{clique $r$-crossing} of order $n$ is the graph obtained from the star $r$-crossing of order $n$
by turning the neighborhood of each root into a clique.
Moreover, we define the \emph{half-graph $r$-crossing} of order $n$ similarly to the star $r$-crossing of order $n$,
where each root~$a_i$ is instead adjacent to $\{ \pi(i',j,1) : i',j \in [n], i \le i' \}$,
and each root $b_j$ is instead adjacent to $\{ \pi(i,j',r) : i,j' \in [n], j \le j' \}$.
Each of the three $r$-crossings contains no edges other than the ones described.
The \emph{comparability grid} of order \(n\) consists of vertices
\(\{ a_{i,j} \mid i,j \in [n] \}\) and edges between vertices \(a_{i,j}\) and \(a_{i',j'}\) if and only if either $i=i'$, or $j=j'$,
or $i<i'\Leftrightarrow j<j'$.

\begin{figure}[h]
    \centering
    \includegraphics[width = \textwidth]{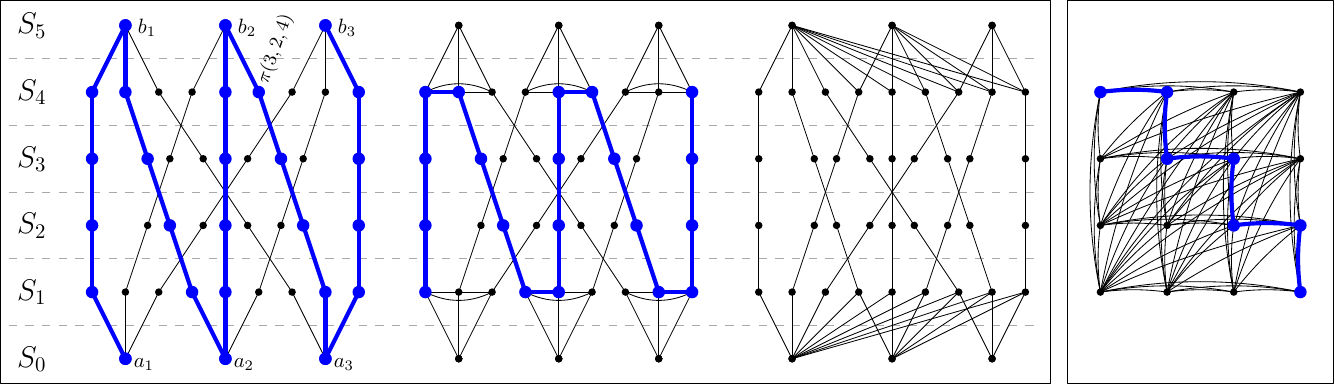}%

    \caption{Left: the star/clique/half-graph $4$-crossing of order $3$. Right: the comparability grid of order $4$.
    Three of the four obstructions contain long induced paths,  marked in blue.
    }
    \label{fig:dependent-patterns}
\end{figure}

We partition the vertices of star, clique, and half-graph $r$-crossings into \emph{slices} $S_0,\dots,S_{r+1}$:
The slice $S_0 := \{a_1,\dots,a_n\}$ consists of the $a$-roots.
The slice $S_t := \{ \pi(i,j,t) : i,j \in [n] \}$, for $t \in [r]$, consists of the $t$th vertices of the paths.
Finally, $S_{r+1} := \{b_1,\dots,b_n\}$ consists of the $b$-roots.


\begin{samepage}
\begin{theorem}[{\cite[Thm.\ 1.6]{dreier2024flipbreakability}}]\label{thm:dependent-forbidden-patterns}
    A graph class $\CC$ is monadically independent if and only if there are $r\geq 1$ and $k \in \N$, such that at least one of the following holds:
    \begin{itemize}
        \item $\CC$ contains a $k$-flip of every star $r$-crossing as an induced subgraph.
        \item $\CC$ contains a $k$-flip of every clique $r$-crossing as an induced subgraph.
        \item $\CC$ contains a $k$-flip of every half-graph $r$-crossing as an induced subgraph.
        \item $\CC$ contains every comparability grid as an induced subgraph.
    \end{itemize}
\end{theorem}
\end{samepage}

The forbidden induced subgraph characterization of monadic dependence 
suggests a clear path towards proving \Cref{thm:2wqo-dependent}. 
We just have to show that none of the four forbidden families is pattern-free.


\begin{proof}[Proof of \Cref{thm:2wqo-dependent}]
    We show the contrapositive: if $\CC$ is not monadically dependent, then $\CC$ is not pattern-free. 
    By \Cref{thm:dependent-forbidden-patterns}, the class of $k$-flips from $\CC$ contains as induced subgraphs all star/clique/half-graph $r$-crossing for some $r \geq 1$, or $\CC$ contains every comparability grid as an induced subgraph.
    Since the hereditary closure of a pattern-free class is pattern-free, 
    and we can apply \Cref{lem:flip-pattern} to undo the flips,    
    it is sufficient to prove that each of the following four classes is not pattern-free:
    \begin{enumerate}
        \begin{multicols}{2}
        \item the class of all star $r$-crossings,
        \item the class of all clique $r$-crossings,
        \item the class of all half-graph $r$-crossings,
        \item the class of all comparability grids.
        \end{multicols}
    \end{enumerate}
    The classes of star and clique crossings are known to contain all paths as induced subgraphs~\cite[Lem.\ 21]{mahlmann2025forbidden}. This also holds for the class of comparability grids.
    To give an intuition\footnote{The only case that is not obvious from the figure is the class of clique $1$-crossings. As noted in \cite{mahlmann2025forbidden}, the hereditary closure of this class contains all \emph{rook graphs}. A layout showing that rook graphs contain long paths is given in \cite[Fig.\ 5]{mahlmann2025forbidden}.}, we marked long paths for each of these three cases in \Cref{fig:dependent-patterns}.
    The conclusion follows from \Cref{lem:paths-not-pattern-free}.

    The last remaining case (half-graph $r$-crossings) is handled by the following claim.

    \begin{restatable}{claim}{lemhgc}\label{lem:hgc}
        For every $r \geq 1$, the half-graph $r$-crossing of order $2mh$ contains an $(m,h)$-pattern.
    \end{restatable}

    \begin{claimproof}[Proof sketch]
        We can find a large pattern inside the half-graph crossings by ``zigzagging'' through it.
        Such an embedding is illustrated in \Cref{fig:hgc-pattern};
        note that it actually exhibits a pattern with even more layers.
        A~formal proof is given in \Cref{sec:appendix-hgc}.
    \end{claimproof}
    This finishes the proof of \Cref{thm:2wqo-dependent}.
\end{proof}

\newcommand{\hgcpathsnippet}[2]{%
    \tikz[baseline=-0.7ex, x=0.42em, y=0.42em]{
        \draw[line width=1.2pt, line join=round, #1] (0,0)--(4,0);
    }%
}
\newcommand{\hgcbluepath}{\hgcpathsnippet{blue!70!white}{}}
\newcommand{\hgcredpath}{\hgcpathsnippet{red!70!white, dash pattern=on 3pt off 2pt}{}}
\newcommand{\hgcvioletpath}{\hgcpathsnippet{violet!70!white, dash pattern=on 0.8pt off 1.2pt}{}}

\begin{figure}[!ht]
    \centering
    \input{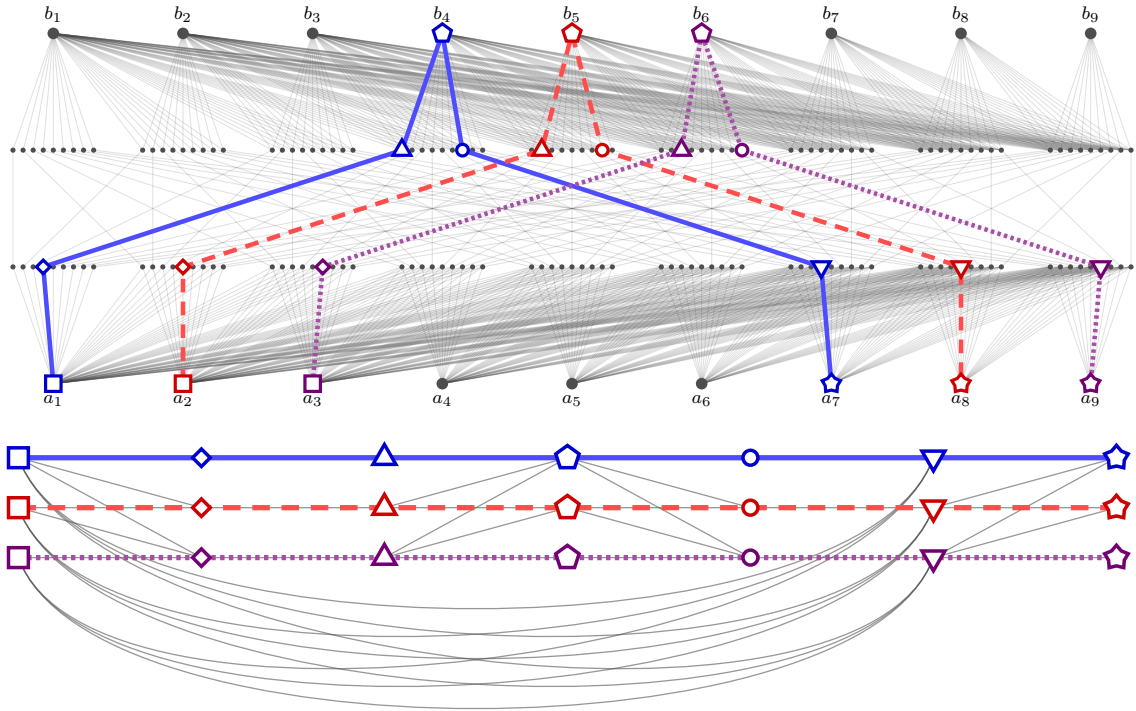}
    \vspace{-\bigskipamount}
    \caption[Pattern embedded in a half-graph crossing]{\emph{Top}: a $(3,7)$-pattern embedded in the half-graph $2$-crossing of order~$9$.
    Vertices of the pattern in the same layer are marked with the same shape.
    Vertices of the pattern with the same first coordinate are joined by one of the paths
    \protect\hgcbluepath, \protect\hgcredpath, or \protect\hgcvioletpath.
    In larger half-graph crossings, the pattern can be continued to the right by going upwards again, i.e., $a_7,a_8,a_9$ would next be routed to $b_{10},b_{11},b_{12}$, and so on.
    \emph{Bottom}: an ``unfolding'' of the induced pattern. In this layout it is easy to see that the highlighted induced subgraph is indeed a pattern.
    }
    \label{fig:hgc-pattern}
\end{figure}

\section{Bounding the rank-width}
\label{sec:final}
In this section we prove our main theorem: every hereditary $2$-WQO graph class has bounded rank-width (equivalently: bounded clique-width).
We have previously established that hereditary $2$-WQO classes are pattern-free, and pattern-free classes are monadically dependent.
For the proof of our main theorem, we distill the relevant structure of pattern-free classes derived from monadic dependence into the following Ramsey-theoretic lemma, whose proof is sketched in \cref{sec:proofsketchpartition}. The full proof can be found in \Cref{sec:appendix-partition}.

\begin{samepage}
\begin{restatable}[Separator Lemma]{lemma}{lemPartition}\label{lem:partition}
    For every pattern-free graph class $\CC$, there are $f \from \N \to \N$ and $k \in \N$ such that for every graph $G \in \CC$ and set $X\subseteq V(G)$ of size at least $f(m)$ there exists partition $P_1,\ldots,P_m$ of $V(G)$ with $m$ parts such that:
    \begin{enumerate}
        \item $X$ intersects each of the sets $P_1,\ldots,P_m$,
        \item for every $1 \leq \alpha \leq m$:
        \[
            \rk_G( L, R) \leq k,
        \]
        where $L := P_{1} \cup \ldots \cup P_{\alpha-1}$ and $R := P_{\alpha+1} \cup \ldots \cup P_{m}$.
    \end{enumerate}
\end{restatable}
\end{samepage}

Informally, the lemma states that in any pattern-free graph class, a large set $X$ can always be ``spread out'' across a partition of the vertices 
into ordered parts $P_1,\ldots,P_m$, such that each part $P_\alpha$ acts as low-rank separator between the parts $P_1,\ldots,P_{\alpha-1}$ left of it and the parts $P_{\alpha + 1},\ldots,P_m$ right of it.
The rank bound~$k$ is a constant depending only on the graph class, not on the graph or the size of $X$. The proof of \Cref{lem:partition} is deferred to \Cref{sec:proofsketchpartition}. We now show how to use it to prove the main result.

We use one more lemma, which is a standard WQO argument,
whose proof is deferred to \Cref{sec:appendix-wl-bounded}.

\begin{restatable}{lemma}{lemWlBounded}\label{lem:wl-bounded}
    For every hereditary $2$-WQO graph class $\CC$, there is a function $f:\N \to \N$ such that for every $k \in \N$ and graph $G \in \CC$, if $G$ contains a set $W \subseteq V(G)$ of size $k$ that is well-linked in $G$, then $G$ contains an induced subgraph $H$ whose vertex set has size at most $f(k)$ and $W \subseteq V(H)$ is also a well-linked set in $H$.
\end{restatable}

We are finally ready to prove the main result of the paper.
We use the small size of $H$ obtained in the lemma above to argue that one of the parts from the partition in \Cref{lem:partition} is disjoint from $H$, effectively separating it.

\begin{theorem}
    Every hereditary $2$-WQO graph class $\CC$ has bounded rank-width.
\end{theorem}
\begin{proof}
    By \Cref{lem:2wqo-pattern-free}, $\CC$ is pattern-free.
    Let $f_\mathsf{part}$ and $k$ be as given by \Cref{lem:partition} for $\CC$,
    and let $f_\mathsf{link}$ be as given by \Cref{lem:wl-bounded}.
    Set $s := k + 1$, $\ell := f_\mathsf{link}(2s)$, and $m := 2s + \ell + 1$.
    We claim that every graph in $\CC$ has rank-width at most $f_\mathsf{part}(m)$.

    Suppose for contradiction that some $G \in \CC$ has rank-width greater than $f_\mathsf{part}(m)$.
    By \Cref{fact:well-linked}, $G$ contains a well-linked set $W$ of size $f_\mathsf{part}(m)$.
    By \Cref{lem:partition}, there is a partition $P_1, \ldots, P_m$ of $V(G)$
    with $P_i \cap W \ne \emptyset$ for each~$i$, and such that for all $1 \le \alpha \le m$,
    \[
        \rk_G\left(L_\alpha, R_\alpha\right) \le k,
    \]
    where $L_\alpha := P_1 \cup \ldots \cup P_{\alpha-1}$ and  $R_\alpha := P_{\alpha+1} \cup \ldots \cup P_m$.

    Pick a sequence $w_1,\ldots, w_m$ satisfying $w_i \in P_i \cap W$ for each $i \in [m]$, and let $W_1$ and $W_2$
    be the first and last $s$ elements of this sequence, respectively.
    Since every subset of a well-linked set is again well-linked~(\Cref{obs:well-linked-hereditary}),
    $W_1 \cup W_2 \subseteq W$ is well-linked in~$G$.
    By \Cref{lem:wl-bounded}, there is an induced subgraph $H$ of $G$
    with $|V(H)| \le \ell$ such that $W_1 \cup W_2$ is contained in $V(H)$ and well-linked in~$H$.

    The parts $P_{s+1}, \ldots, P_{s+\ell+1}$ lie between $W_1$ and~$W_2$ (see \cref{fig:bounding-rankwidth}).
    Since there are $\ell + 1 > |V(H)|$ such parts,
    by pigeonhole some $P_\alpha$ among them satisfies $P_\alpha \cap V(H) = \emptyset$.

    Set $L := L_\alpha$ and $R:= R_\alpha$.
    By the choice of indices, $W_1 \subseteq L$ and $W_2 \subseteq R$.
    Set $Z := L \cap V(H)$.
    Since $P_\alpha \cap V(H) = \emptyset$,
    we have $V(H) \setminus Z = R \cap V(H)$, so
    \[
        \rk_H(Z)
        = \rk_H\!\big(L \cap V(H),\; R \cap V(H)\big)
        \le \rk_G\!\left(L,\; R\right)
        \le k,
    \]
    where the first inequality holds because the adjacency matrix of $H$ between these two sets is a submatrix of the corresponding matrix of $G$.
    On the other hand, $W_1 \subseteq Z \subseteq V(H) \setminus W_2$,
    so the well-linkedness of $W_1 \cup W_2$ in $H$ applied to the bipartition $(W_1, W_2)$ gives the following contradiction
    \[
        k\ge \rk_H(Z) \ge \min(|W_1|, |W_2|) = s = k + 1.\qedhere
    \]
\end{proof}

\vspace{-1em}
\begin{figure}[h!]
\centering
\includegraphics[scale = 0.85]{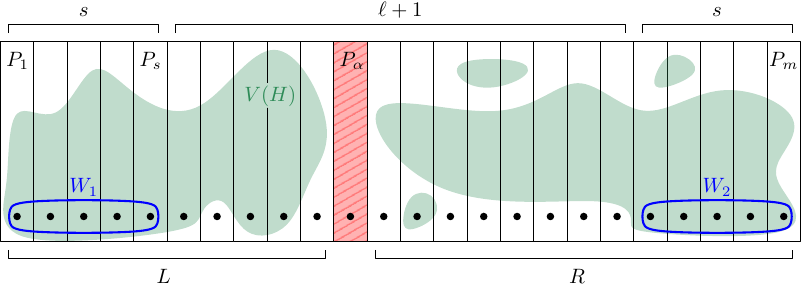}
\vspace{-1em}
\caption{The partition $P_1,\ldots,P_m$ into $m = 2s + \ell + 1$ parts.
    The sets $W_1$ and $W_2$ each have size~$s$.
    At least one part $P_\alpha$ is disjoint from $V(H)$,
    yielding the decomposition into $L$ and $R$ with $W_1 \subseteq L$ and $W_2 \subseteq R$.}
\label{fig:bounding-rankwidth}
\end{figure}

\section{Proof sketch of the Separator Lemma}\label{sec:proofsketchpartition}
In this section, we sketch the proof of the Separator Lemma,  which we restate here for convenience.

\lemPartition*

Instead of proving the Separator Lemma directly, we will derive it from the \emph{Double Separator Lemma} (\Cref{lem:partitionDouble}), which we obtain by replacing the second condition of the Separator Lemma with:

\emph{\begin{enumerate}
    \item[2.] for all $1\leq \alpha < \beta \leq m$:
        \[
            \rk_G\left( M, L \uplus  R  \right) \leq k,
        \]
        where $L := P_{1} \cup \ldots \cup P_{\alpha-1}$,  $M := P_{\alpha+1} \cup \ldots \cup P_{\beta-1}$, and $R :=  P_{\beta+1} \cup \ldots \cup P_m$.
\end{enumerate}
}

In this modified version we demand that any two parts $P_\alpha$ and $P_\beta$ act as a low rank separator between the parts between them (i.e., the middle $M$) and the rest of the parts (i.e., the left and right parts $L \uplus R$).
The Double Separator Lemma easily implies the Separator Lemma.

\begin{proof}[Proof of the Separator Lemma using the Double Separator Lemma]
    See \Cref{fig:partition-wrap} for an illustration.
    Applying the Double Separator Lemma for $2m$ yields a partition $Q_1,\ldots, Q_{2m}$.
    We obtain the partition $P_1,\ldots, P_m$ for the Separator Lemma by ``folding'' the previous partition setting $P_i := Q_i \cup Q_{2m+1-i}$.
    We can now deduce the separation property for each $P_\alpha$ from the Double Separation Lemma for $Q_\alpha$ and $Q_\beta$ with $\beta := 2m+1-\alpha$.
\end{proof}

\begin{figure}[htbp]
    \centering
    \includegraphics{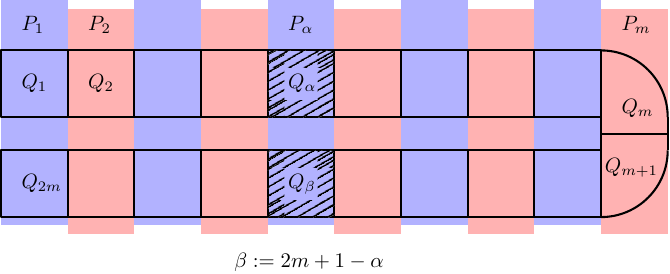}
    \caption{Visual proof of the Separator Lemma using the Double Separator Lemma.}
    \label{fig:partition-wrap}
\end{figure}

Hence, it only remains to prove the Double Separator Lemma.
We give a sketch of the proof here and defer the full proof to \Cref{sec:appendix-partition}.

\begin{figure}[htbp]
    \centering
    \includegraphics[width = 0.8\linewidth]{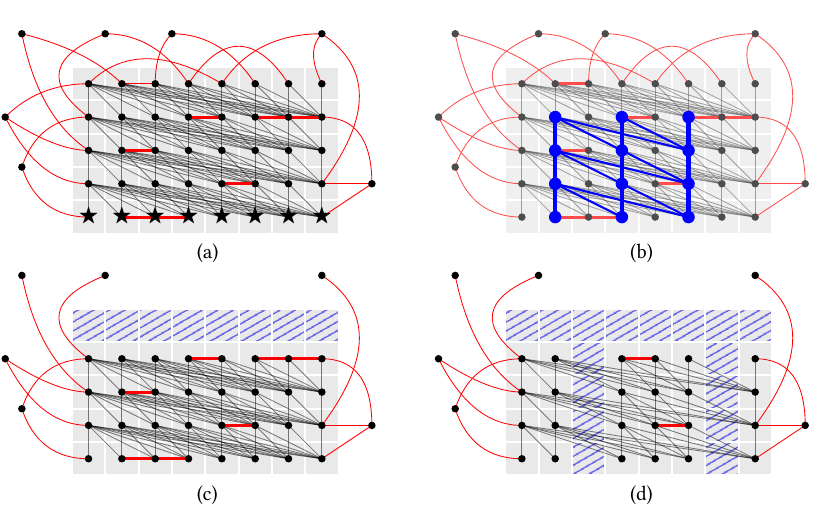}
    \caption{An illustration of the proof of \Cref{lem:partition}.
    Each panel shows the grid layout of an insulator.
    The vertices in the bottom row are the vertices of $A$ (marked with stars in the first panel). The height $h$ is the number of rows.
    }
    \label{fig:partition-sketch}
\end{figure}

\begin{proof}[Proof sketch of the Double Separator Lemma]
Fix a large vertex set $X \subseteq V(G)$ as in the statement of the lemma.
We have previously established that pattern-free classes are monadically dependent (\Cref{thm:2wqo-dependent}).
Hence, we can make use of the \emph{insulation property}~\cite{dreier2024flipbreakability}, a structural characterization of monadic dependence.
The insulation property guarantees that for every fixed $h\in\N$, we can find an \emph{insulator} of \emph{height} $h$ guarding a large subset of $A \subseteq X$. (We will later choose $h$ depending on the size of the forbidden pattern.)

An insulator is an $O_h(1)$-flip $H$ of the graph $G$, in which the neighborhood of the set $A$ is well-controlled.
An example is depicted in \Cref{fig:partition-sketch} (a).
The vertices of $A$ are all marked with a star symbol. In the insulator, the neighborhood of $A$ is organized in a grid like fashion.
Each column of the grid contains a vertex of $A$ in its bottom cell, and the number of rows of the grid is equal to the height $h$ of the insulator.
The union of the leftmost column, the rightmost column, and the top row of the grid form its \emph{boundary}, while the rest of the grid is called the \emph{interior}.
The insulator controls the interior of the grid as follows. In the graph $H$:
\begin{enumerate}
    \item No vertex outside the grid is adjacent to any interior vertex.
    \item Interior vertices in non-consecutive rows are non-adjacent.
    \item There is an $O_h(1)$-coloring of $V(H)$ such that adjacency between two interior vertices $u$ and $v$ in non-consecutive columns depends only on their colors and on the left-to-right order of $u$ and $v$ in the grid.
    \item Every interior vertex has
    \begin{itemize}
        \item at least one neighbor in the cell directly below it, and
        \item no neighbors in the row below that are strictly to its left.
    \end{itemize}
\end{enumerate}
A complete formal definition of an insulator is given in \Cref{def:insulator} in the appendix.
Crucially the insulator controls all edges except:
\begin{itemize}
    \item edges from outside vertices towards the boundary, and
    \item edges between vertices in neighboring cells.
\end{itemize}
We have marked such \emph{exceptional edges} in red in \Cref{fig:partition-sketch} (a).
Intuitively, the greater we choose the height~$h$ of the insulator, the farther we push away the boundary from the set $A$, and the more control we gain.
(The cost of this is the rising complexity of the $O_h(1)$-flip $H$ and the $O_h(1)$-coloring.)

So far, we have just applied the insulation property.
The key insight is that inside an insulator of height~$h$, if many of its top cells are occupied, then we find an $(\ell,h - 1)$-pattern, where $\ell$ is proportional to the size of $A$. Such a pattern is highlighted in \Cref{fig:partition-sketch} (b).
The layers of the pattern embed into the rows of the insulator, and we choose the pattern in non-consecutive columns.
By pigeonholing over the parts of the $O_h(1)$-flip and over the $O_h(1)$ colors, we assure that each layer is a clique or an independent set.

We assume $A$ to be large and the graph class that we work in to be pattern-free.
Hence, we know that for some fixed $h \in \N$, we can build insulators in which most of the top cells must be empty. For this proof sketch we assume all the top cells are empty. (Otherwise, we have a large pattern, as argued above.)
This is depicted in \Cref{fig:partition-sketch} (c).
Now vertices from outside the grid can only connect to the leftmost and rightmost column of the insulator.
The interior columns $C_1,\ldots,C_m$ of this insulator form the partition $P_1,\ldots,P_m$ of $V(G)$ (each column forms its own part), where we put all the vertices from outside the interior into the part $P_1$.
By construction, each part contains a vertex from $X$.
Now if we ignore any two columns/parts $P_\alpha$ and $P_\beta$, this cuts out a middle part $M := P_{\alpha + 1} \cup \ldots \cup P_{\beta - 1}$ that is completely in the interior of the insulator and whose neighboring columns are exactly the ignored parts $P_\alpha$ and $P_\beta$.
Therefore, the adjacency of $L \uplus R := V(G) - (M \cup P_\alpha \cup P_\beta)$ towards $M$ is completely controlled by the insulator.
This is visualized in \Cref{fig:partition-sketch} (d), where no exceptional (red) edges run between the two sets.
This bounds the rank and finishes the sketch of the proof of \Cref{lem:partition}.
\end{proof}

\paragraph{AI Disclosure.}
We used Claude Code with Opus 4.6 and ChatGPT Pro 5.4  for proofreading, refactoring, improving the presentation of the text, and filling in details of proofs; all mathematical results are due to the authors.
The authors verified the correctness and originality of all content including references.

\bibliographystyle{plain}
\bibliography{ref}

\appendix
\section{Proof of \Cref{lem:flip-pattern}}\label{sec:appendix-flips}

\flippattern*

\begin{proof}
    Let $G = H \oplus F$ be a $k$-flip of a $(k^r \cdot m, r)$-pattern $H$,
    defined by a partition $\KK$ of $V(H)$ into at most $k$ parts.
    Assign to each index $i \in [k^r \cdot m]$ its \emph{color sequence}
    $\mathbf{C}(i) := (\KK((i,1)), \ldots, \KK((i,r))) \in \KK^r$.
    Since there are at most $k^r$ color sequences,
    by pigeonhole there exist $m$ indices $i_1 < \cdots < i_m$
    sharing the same sequence $(C_1, \ldots, C_r)$.

    Let $P$ be the subgraph of $G$ induced on $\{(i_\ell, j)\}$,
    with vertices relabeled $(\ell, j)$ order-preservingly.
    All three pattern conditions are satisfied:
    for each layer $j$, every vertex in $L_j(P)$ has color $C_j$,
    so the flip acts uniformly on every pair of vertices from the same layer
    and on every pair from non-consecutive layers,
    preserving the clique/independent-set and full-homogeneity conditions.
    For consecutive layers $j, j{+}1$, the flip uniformly complements
    (or preserves) the comparison relation $\sim_H$ governing the adjacency between $L_j(H)$ and $L_{j+1}(H)$.
    Since the relabeling is order-preserving and the negation of any relation in
    $\{=,\neq,\leq,\geq,<,>\}$ is again in that set, the result is still a comparison relation.
\end{proof}

\section{Proof of \Cref{lem:hgc}}\label{sec:appendix-hgc}

\lemhgc*

\begin{proof}
    Let $G$ be the half-graph $r$-crossing of order $2mh$ and fix its partition into slices $S_0,\ldots,S_{r+1}$.
    Let
    \[
        I_1<J_1<I_2<J_2<\dots<I_h<J_h
    \]
    be the partition of $[2mh]$ into $2h$ many intervals of size $m$.
    For each $j\in[h]$, fix the two unique order preserving bijections
    \[
        \alpha_j\colon [m]\to I_j, \qquad \beta_j\colon [m]\to J_j.
    \]

    Recall that in the half-graph $r$-crossing of order $n := 2mh$, the root vertices are called $a_1,\ldots,a_n$ and $b_1,\ldots,b_n$, and the $t$th vertex on the path connecting $a_i$ and $b_j$ is called $\pi(i,j,t)$~(cf.\ \Cref{fig:dependent-patterns}).

    For each $j\in[h]$ and $t\in[r]$, define
    \[
        M^t_j:=\{\pi(\alpha_j(i),\beta_j(i),t):i\in[m]\} \subseteq S_t.
    \]
    For each $j\in[h-1]$ and $t\in[r]$, define also
    \[
        N^t_j:=\{\pi(\alpha_{j+1}(i),\beta_j(i),t):i\in[m]\} \subseteq S_t.
    \]
    For each $j\in[h]$, define
    \[
        A_j:=\{a_{\alpha_j(i)}:i\in[m]\} \subseteq S_0,
        \quad
        B_j:=\{b_{\beta_j(i)}:i\in[m]\} \subseteq S_{r+1}.
    \]
    Consider the sequence
    \[
        A_1,M^1_1,\dots,M^r_1,B_1,N^r_1,\dots,N^1_1,A_2,M^1_2,\dots
    \]
    of layers.
    We denote the first $h$ elements of this sequence by
    \[
        L_1,L_2,L_3, \ldots, L_h.
    \]
    Let $H$ be the induced subgraph on the union of these $h$ layers.
    We claim that $H$ is an $(m,h)$-pattern. Let us check the required conditions.

    \paragraph{Condition 1: Each layer induces a clique or independent set.}
    Each constructed layer is fully contained in a single slice of $G$.
    By definition of a half-graph crossing, each slice forms an independent set.
    Thus, each layer induces an independent set.

    \paragraph{Condition 2: Non-consecutive layers are fully homogeneous.}
    Fix two non-consecutive layers $L_j$ and $L_{j'}$ with $j + 1 < j'$ contained in slices $S_t \supseteq L_j$ and $S_{t'} \supseteq L_{j'}$ of the crossing.

    \begin{itemize}
    \item \emph{Case 1: $t = t'$ or $|t - t'| > 1$.}
    The half-graph crossing has no edges within a slice (as established in Condition~1) and no edges between non-consecutive slices.
    Hence $L_j$ and $L_{j'}$ are fully non-adjacent.

    \item \emph{Case 2: $\{t,t'\} = \{t, t+1\}$ for some $t \in [r-1]$, i.e., both slices are internal.}
    Every edge of the crossing between $S_t$ and $S_{t+1}$ connects $\pi(a,b,t)$ to $\pi(a,b,t+1)$ for the same pair $(a,b) \in [n]^2$.
    Each layer carries a \emph{pair-index set}: the set of pairs $(a,b)$ indexing its vertices. For each $k\in[h]$
    \begin{itemize}
        \item the pair-index set of $M^t_k$ is $\{(\alpha_k(i), \beta_k(i)) : i \in [m]\} \subseteq I_k \times J_k$,
        \item the pair-index set of $N^t_k$ is $\{(\alpha_{k+1}(i), \beta_k(i)) : i \in [m]\} \subseteq I_{k+1} \times J_k$.
    \end{itemize}
    The only pairs of layers in consecutive internal slices $S_t, S_{t+1}$ sharing a pair-index set are of the form $(M^t_k, M^{t+1}_k)$ or $(N^{t+1}_k, N^t_k)$ for some $k\in[h]$. These are precisely the consecutive pairs in our layer sequence.
    For any other combination, the pair-index sets are disjoint: 
    mixing $M$- and $N$-type layers from the same group $k$ results in different first-coordinate intervals (i.e., $I_k \cap I_{k+1} = \emptyset$), while
    layers from different groups $k \neq k'$ use different second coordinates (i.e., $J_k \cap J_{k'} = \emptyset$).
    Hence, $L_j$ and $L_{j'}$ have no edges between them, so they are fully non-adjacent.

    \item \emph{Case 3: $\{t,t'\} = \{0,1\}$, i.e., one layer lies in $S_0$ (containing $a$-roots) and the other in $S_1$.}
    The only non-consecutive pairs of layers $(L_j,L_{j'})$ involving both slices $S_0$ and $S_1$ are of one of the following four forms.
    In all cases, we can assume $k \leq k'$, because $L_j$ comes before $L_{j'}$.

    \begin{itemize}
        \item $(A_k,M_{k'}^1)$ with $k < k'$:
        Recall that  
        \[
            A_k = \{a_{\alpha_k(i)}:i\in[m]\}
            \quad \text{and} \quad
            M_{k'}^1 = \{\pi(\alpha_{k'}(i),\beta_{k'}(i),1):i\in[m]\}
        \]
        and that the neighbors of $a_i$ are exactly $\{\pi(i',j,1) \mid i \leq i' \leq n \text{ and } j \in [n] \}$.
        Because all elements from the image $I_k$ of $\alpha_k$ are smaller than from the image $I_{k'}$ of $\alpha_{k'}$, we have that $L_j = A_k$ and $L_{j'} = M_{k'}^1$ are fully adjacent.

        \item $(A_k,N_{k'}^1)$ with $k \leq k'$: As in the previous case, the two layers are fully adjacent. This time we use the fact that (even if $k = k'$) the image of $\alpha_{k}$ is smaller than the image of $\alpha_{k' + 1}$, which is used to define $N_{k'}^1$. 

        \item $(M_{k}^1, A_{k'})$ with $k < k'$: Analogous to the previous cases, but now the two layers are fully non-adjacent, because the image of $\alpha_{k'}$ is larger than the image of $\alpha_k$.

        \item $(N_{k}^1,A_{k'})$ with $k < k' - 1$:
        Here we have the stronger assumption $k < k' - 1$, because if $k = k' - 1$, then the two layers would be consecutive. Analogous to the previous cases, the two layers are fully non-adjacent, because the image of $\alpha_{k'}$ is larger than the image of $\alpha_{k+1}$.
    \end{itemize}

    \item \emph{Case 4: $\{t,t'\} = \{r,r+1\}$, i.e., one layer lies in $S_{r+1}$ (containing $b$-roots) and the other in $S_r$.}
    Analogous to Case 3.
    
    \end{itemize}
    Having exhausted all cases, this proves the second condition.

    \paragraph{Condition 3: Consecutive layers are connected by a comparison relation.}
    For the pairs
    \[
        M^t_j,M^{t+1}_j \quad (1\le t<r),
        \qquad
        N^{t+1}_j,N^t_j \quad (1\le t<r),
    \]
    adjacency in the crossing is exactly the equality relation on the parameter $i$.
    The boundary pairs
    \[
        A_j,M^1_j,\qquad M^r_j,B_j,\qquad B_j,N^r_j,\qquad N^1_j,A_{j+1}
    \]
    realize the relations
    \[
        \le,\ \ge,\ \le,\ \ge
    \]
    in the half-graph crossing.
    In each case, the adjacency between consecutive layers is determined by a comparison relation on the index $i \in [m]$, establishing the third condition. This finishes the proof that $H$ is an $(m,h)$-pattern.
\end{proof}
\section{Proof of \Cref{lem:wl-bounded}}\label{sec:appendix-wl-bounded}

\lemWlBounded*

\begin{proof}
    Call a pair $(G, W)$ \emph{$k$-minimal} if $G \in \CC$, $W \subseteq V(G)$ is a well-linked set of size $k$,
    and no proper induced subgraph of $G$ containing $W$ has $W$ as a well-linked set.

    Suppose for contradiction that there exist infinitely many $k$-minimal pairs $(G_1,W_1),(G_2,W_2),\ldots$
    with $|V(G_1)| < |V(G_2)| < \cdots$.
    Label each $G_i$ by $\ell_i : V(G_i) \to [2]$ with $\ell_i(v) = 1$ if $v \in W_i$
    and $\ell_i(v) = 2$ otherwise.
    Since $\CC$ is $2$-WQO, there exist $i < j$ and an injective label-preserving
    induced-subgraph embedding $\phi : V(G_i) \to V(G_j)$.
    Label-preservation gives $\phi(W_i) \subseteq W_j$,
    and since $\phi$ is injective and $|W_i| = |W_j| = k$, we have $\phi(W_i) = W_j$.
    Because $\phi$ is an embedding, $G_j[\phi(V(G_i))]$ is isomorphic to $G_i$,
    so $W_j = \phi(W_i)$ is well-linked in $G_j[\phi(V(G_i))]$.
    Since $|V(G_i)| < |V(G_j)|$, the subgraph $G_j[\phi(V(G_i))]$ is a proper
    induced subgraph of $G_j$ that contains $W_j$, contradicting $k$-minimality of $(G_j,W_j)$.
    Hence there is a bound $f(k)$ such that every $k$-minimal graph has at most $f(k)$ vertices.

    The statement of the lemma now follows by choosing $H \in \CC$ to be any induced subgraph of $G$ such that $(H,W)$ is $k$-minimal.
\end{proof}

\section{Proof of the Double Separator Lemma}\label{sec:appendix-partition}

We prove the Double Separator Lemma by building on the theory of insulators from~\cite{dreier2024flipbreakability}.
We first recall the relevant definitions (\Cref{subsec:insulators}), then establish a key lemma guaranteeing the existence of an insulator with many empty top cells (\Cref{subsec:bounded-height-insulators}).

\subsection{Insulators}\label{subsec:insulators}

We recall the notion of the insulation property from (the full version of)~\cite{dreier2024flipbreakability}.
We mostly replicate the required definitions from~\cite{dreier2024flipbreakability}, but make slight simplifications when applicable to our setting, as made explicit in the upcoming \Cref{rem:weaken}.

\begin{definition}[Grids]\label{def:grids}
    Fix a non-empty sequence $I$ and an integer $h \ge 1$.
    A \emph{grid} $A$ \emph{indexed by} $I$ and of \emph{height} $h$ in a graph $G$ is a collection of pairwise disjoint sets $A[i,r] \subseteq V(G)$, for $i\in I$ and $r \in [h]$, called \emph{cells}.
\end{definition}

To facilitate notation, we often assume, up to renaming, the indexing sequence to be   $I= (1, \ldots, n)$.
We do so in the following.
For subsets $J \subseteq I$ and $R \subseteq [h]$, we write
$A[J,R] = \bigcup_{i \in J, r \in R} A[i,r]$.
We often use implicitly defined sets via wildcards and comparisons. For example $A[{\le}i,*]$ stands for $A\bigl[\{1,\dots,i\},[h]\bigr]$.
In particular, we use $A[i,*]:=\bigcup_{r\in [h]} A[i,r]$ and $A[*,r]:=\bigcup_{i\in I}A[i,r]$,
and refer to those sets as to \emph{columns} and \emph{rows} of $A$, respectively.
In slight abuse of notation, we often write $A$ instead of $A[*,*]$ to denote the set of all elements inside the grid.
We define the \emph{interior} of $A$ as $\int(A) := A \setminus (A[1,*] \cup A[n,*] \cup A[*,h])$.
It is intentional and will later be important that the cells $A[2,1], \ldots, A[n-1,1]$ of the bottom row are part of the interior.
Moreover, we say two columns $A[i,*]$ and $A[j,*]$  are \emph{close}, if $|i-j| \le 1$
and two rows $A[*,r]$ and $A[*,t]$ are \emph{close}, if $|r-t| \le 1$.
Two cells are \emph{close} if their respective columns and rows are close.
Note that unlike standard matrix notation, to highlight the hierarchical relationship between columns and rows,
our notation $A[i,r]$ first mentions the column index $i \in I$ and then the row index~$r \in [h]$.

\newcommand{\explanation}{$\blacktriangleright$\ }
\begin{definition}[Insulators]\label{def:insulator}
    An \emph{insulator} $\gc{A} = (A, \KK, F, R)$ \emph{indexed by a sequence $I$} of \emph{height} $h$ and \emph{cost} $k$ in a graph $G$ consists of
    \begin{itemize}
        \item a grid $A$ indexed by $I$ and of height $h$,
        \item a partition $\KK$ of $V(G)$ into at most $k$ color classes,
        \item a symmetric relation $F\subseteq \KK^2$ specifying a flip $G' := G \oplus F$,
        \item a relation $R \subseteq \KK^2$,
    \end{itemize}
    such that all the following conditions hold:
    \begin{enumerate}[leftmargin= 4em, label={(I.$\arabic*$)}] 
        \item\label{itm:rootedness} Every vertex $v \in A[i,r]$ with $r>1$, $i \in I$ has a neighbor in the cell $A[i,r-1]$ in $G'$. 

        \smallskip
        \explanation The mandatory downward edge, together with \ref{itm:adj-bot-left}, keeps each column cohesive.

        \item\label{itm:outside} 
            For every $v \notin A[*,*]$ and $X \in \KK$ we require that $v$ is \emph{homogeneous} to $X \cap \int(A)$ in $G$ \\
            (that is, either all or no vertices in $X \cap \int(A)$ are adjacent to $v$).

            \smallskip
            \explanation The inside of the insulator is ``insulated'' from its outside:
            the adjacencies between the two are described using only colors.

        \item\label{itm:adjacency} 
            For every $u \in A[i,r]$ with $r < h$, $i \in I$ and $v \in A$ we have the following: \\
            (Up to renaming, we assume $I = (1,\ldots,n)$.)
            \begin{enumerate}[leftmargin= 5em, label={(I.$3$.$\arabic*$)}]
                \item \label{itm:adj-different-rows} If $u$ and $v$ are in rows that are not close and $u\in\int(A)$, \\
                    then they are non-adjacent in $G'$.

                \item \label{itm:adj-bot-left} If $v \in A[{<}i,r-1] \cup A[{>}i,r+1]$, then $u$ and $v$ are non-adjacent in $G'$.
                \item \label{itm:adj-left} If $v \in A[{>}i+1, \{r,r-1\}]$, then $uv \in E(G)$ if and only if $(\KK(u),\KK(v))\in R$.
                \item \label{itm:adj-right} If $v \in A[{<}i-1, \{r,r+1\}]$, then $uv \in E(G)$ if and only if $(\KK(v),\KK(u))\in R$.
            \end{enumerate}
            Otherwise, we make no claims regarding the adjacency of $u$ and $v$.

            \smallskip
            \explanation 
            Properties \ref{itm:adj-different-rows}, \ref{itm:adj-left}, and \ref{itm:adj-right} provide vertical and horizontal insulation inside the insulator.
            Property \ref{itm:adj-bot-left} helps to keep each column cohesive.
            See \Cref{fig:adjacency} for an illustration.

\end{enumerate}
\end{definition}

\begin{remark}\label{rem:weaken}
    For the sake of presentation, we have relaxed the definition of an insulator. The definition given in~\cite{dreier2024flipbreakability} contains some more restrictions, which we do not require for our arguments to work.
    In particular, two variants of insulators, \emph{unordered} and \emph{ordered} ones, are defined in~\cite{dreier2024flipbreakability}. As both of them satisfy the subset of properties we have restricted ourselves to, we can treat them uniformly.
\end{remark}

\begin{figure}[htbp]
    \center
    \includegraphics[width = \linewidth]{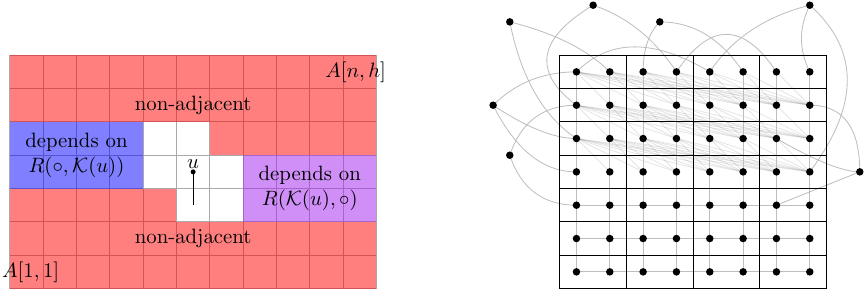}%
    \caption{To the left: an illustration of how the insulator property \ref{itm:adjacency} controls the adjacency of a vertex $u$ within the insulator. To the right: an example of an insulator.}\label{fig:adjacency}
\end{figure}

An example of an insulator is depicted in \cref{fig:adjacency}.


\begin{definition}
    Let $\gc A$ be an insulator with grid $A$ indexed by a sequence $I$ in a graph $G$ and let $W \subseteq V(G)$.
    We say that $\gc A$ \emph{insulates} $W$
    if there is a bijection $f: W \rightarrow I$, such that for all $v \in W$
    \[
        v \in A[f(v),1].
    \]
    A set $W$ is \emph{$(r,k)$-insulated} in $G$ if there is an insulator $\gc A$ of height $r$ and cost $k$ that insulates $W$.
\end{definition}

\begin{definition}\label{def:insulationProperty}
    A graph class $\CC$ has the \emph{insulation property}, if for every radius $r\in \N$ there exist a function $N_r:\N\rightarrow\N$ and a constant $k_r \in \N$ such that for every $m\in\N$, $G\in \CC$, $W\subseteq V(G)$ with $|W| \geq N_r(m)$, there is a subset $W_\star\subseteq W$ of size at least $m$ that is $(r,k_r)$-insulated in $G$.
\end{definition}

\begin{theorem}[{\cite[Thm.\ 13.1]{dreier2024flipbreakability}}]\label{thm:insulation-property}
    Every monadically dependent graph class has the insulation property.
\end{theorem}
\begin{remark}
    In~\cite{dreier2024flipbreakability}, the above \Cref{thm:insulation-property} is stated as an exact characterization: also every class with the insulation property is monadically dependent.
    Since we have weakened the definition of an insulator (cf.\ \Cref{rem:weaken}), this also weakens the insulation property, so we only preserve one implication.
    We remark (but we do not prove, and do not require) that this weaker insulation property still characterizes monadic dependence.
\end{remark}

\subsection{Bounded height insulators}\label{subsec:bounded-height-insulators}

\begin{lemma}\label{lem:bounded-height}
    Let $G$ be an $(m,r)$-pattern-free graph. Every insulator $\gc A$ of height $r+1$ and cost $k$ in $G$ has fewer than $N(m,r,k) := 2k^r \cdot m + 2$ many columns whose top cell is non-empty.

    \smallskip\noindent
    More precisely, if the grid $A$ of $\gc A$ is indexed by $I$, then
    \[
        \bigl|\{i \in I : A[i,r+1] \neq \emptyset\}\bigr| < N(m,r,k).
    \]
\end{lemma}

\begin{proof}
    Assume that $I = (1,\ldots,n)$ and
    assume towards a contradiction that there are at least $N(m,r,k)$ many columns $I_1 := \{i \in I : A[i,r+1] \neq \emptyset\}$ with a non-empty top cell.
    We construct an $(m,r)$-pattern as an induced subgraph of $G$, contradicting the assumption.
    Write $\gc A = (A, \KK, F, R)$ and $G' := G \oplus F$.

    \paragraph{Step 1: Selecting well-spaced columns.}
    By removing the two outermost columns of $I_1$, and then keeping only every second column, we find a subset $I_2 \subseteq I_1$ of size 
    \[
        |I_2| \geq \frac{|I_1|-2}{2} \geq k^r \cdot m    
    \]
    of pairwise non-close columns, that contains no boundary column.

    \paragraph{Step 2: Constructing downward paths.}
    For each $i \in I_2$, we build a path through column $i$ from top to bottom.
    Pick an arbitrary vertex $v(i,r+1) \in A[i,r+1]$ (non-empty since $i \in I_1$).
    For each $j \in [r]$, use the insulator property \ref{itm:rootedness} to pick a neighbor $v(i,j)\in A[i,j]$ of $v(i,j+1)$ in the graph $G'$.
    In the following we will only use $v(i,1), \ldots, v(i,r)$, that crucially all lie in the interior of $A$.

    \paragraph{Step 3: Pigeonholing on color sequences.}
    We assign to each path a \emph{color sequence}
    \[
     \big(\KK(v(i,1)), \ldots, \KK(v(i,r))\big) \in \KK^r.
    \]
    Since $|\KK| \le k$, there are at most $k^r$ distinct color sequences.
    By the pigeonhole principle, there is a subset $I_3 \subseteq I_2$ with $|I_3| \ge m$ such that all paths indexed by $I_3$ share the same color sequence $(C_1, \ldots, C_r)$.

    Write $I_3 = \{\alpha(1) < \alpha(2) < \cdots < \alpha(m)\}$ and set $w(i,j) := v(\alpha(i),j)$ for $i \in [m]$ and $j \in [r]$.

    \paragraph{Step 4: Verifying the pattern.}
    We claim that $\{w(i,j) : (i,j) \in [m] \times [r]\}$ induces an $(m,r)$-pattern in~$G$, with monochromatic layers 
    \[
        L_j := \{w(1,j), \ldots, w(m,j)\} \subseteq C_j \in \KK   \quad \text{for $j \in [r]$.}
    \]
    We verify the required conditions.
    \paragraph{Condition 1: Each layer induces a clique or independent set.}

    Consider two vertices $w(i,j)$ and $w(i',j)$ with $i < i'$ contained in the same layer $L_j$. By construction, the two vertices are not contained in top row of $A$, and are not in two columns that are close in $A$. Hence, the insulator property~\ref{itm:adj-left} guarantees:
    \[
        w(i,j)w(i',j) \in E(G) \quad\Leftrightarrow\quad (C_j, C_j) \in R.
    \]
    This is independent of $i$ and $i'$, so $L_j$ is a clique or independent set.

    \paragraph{Condition 2: Non-consecutive layers are fully homogeneous.}

    Consider two non-consecutive layers $L_j$ and $L_{j'}$, i.e., $|j-j'| > 1$. By construction, the two layers contain only vertices from the interior of $A$, from the two non-close rows $L_j \subseteq A[*,j]$ and $L_{j'}\subseteq A[*,j']$. Hence, the insulator property~\ref{itm:adj-different-rows} guarantees that $L_j$ and $L_{j'}$ are non-adjacent in the $\KK$-flip $G'$.
    Now if $(C_j,C_{j'}) \notin F$ then the same holds in $G$, otherwise $(C_j,C_{j'}) \in F$ and $L_j$ and $L_{j'}$ are fully adjacent in $G$.

    \paragraph{Condition 3: Consecutive layers are connected by a comparison relation.}

    Fix $j \in [r-1]$ and define booleans
    \[
        p_j := \bigl[(C_j, C_{j+1}) \in F\bigr], \qquad q_j := \bigl[(C_{j+1}, C_j) \in R\bigr].
    \]
    We show that whether $w(i,j)w(i',j+1) \in E(G)$ depends only on the order of $i$ and $i'$:

    \begin{itemize}
    \item \emph{Case $i < i'$:}
    Since $\alpha(i) < \alpha(i')$, we have $w(i',j+1) \in A[{>}\alpha(i), j+1]$.
    By~\ref{itm:adj-bot-left} applied to $u = w(i,j)$, the two vertices are non-adjacent in~$G'$.
    Hence,
    \[
        w(i,j)w(i',j+1) \in E(G)\quad \Leftrightarrow \quad p_j.    
    \]
    
    \item \emph{Case $i = i'$:}
    By the path construction, $w(i,j)$ and $w(i,j+1)$ are adjacent in the flip~$G'$.
    Hence, 
    \[
        w(i,j)w(i,j+1) \in E(G) \quad \Leftrightarrow \quad \neg p_j.
    \]

    \item \emph{Case $i > i'$:}
    Since $\alpha(i) - 1 > \alpha(i')$ (non-close columns), we have $w(i',j+1) \in A[{<}\alpha(i)-1, \{j, j+1\}]$.
    By~\ref{itm:adj-right} applied to $u = w(i,j)$, the adjacency in $G$ is determined by the colors:
    \[
        w(i,j)w(i',j+1) \in E(G)
        \quad \Leftrightarrow \quad (C_{j+1}, C_j) \in R  \quad \Leftrightarrow \quad q_j.
    \]
    \end{itemize}
    In summary, $w(i,j)w(i',j+1) \in E(G)$ iff $p_j$ when $i < i'$, iff $\neg p_j$ when $i = i'$, and iff $q_j$ when $i > i'$.
    Each of the four combinations of $(p_j, q_j)$ yields one of the six comparison relations:
    \begin{align*}
        p_j = 0,\; q_j = 0 &\colon \quad \text{edge iff $i = i'$,} \\
        p_j = 1,\; q_j = 1 &\colon \quad \text{edge iff $i \neq i'$,} \\
        p_j = 1,\; q_j = 0 &\colon \quad \text{edge iff $i < i'$,} \\
        p_j = 0,\; q_j = 1 &\colon \quad \text{edge iff $i \ge i'$.}
    \end{align*}

    We conclude that $\{w(i,j)\}$ induces an $(m,r)$-pattern in $G$, contradicting the assumption that $G$ is $(m,r)$-pattern-free.
\end{proof}

\begin{lemma}\label{lem:empty-top-insulator}
    For every pattern-free graph class $\CC$, there exist a function $f \from \N \to \N$ and integers $h,q \in \N$ such that for every graph $G \in \CC$ and every set $X \subseteq V(G)$ with $|X| \ge f(m)$,
    there is an insulator $\gc A$ of height $h$ and cost $q$ in $G$ that insulates a set $W \subseteq X$ of size at least $m$ and whose indexing sequence $(a_1,\ldots,a_n)$ contains $m$ consecutive columns
    \[
        a_{t+1},\ldots,a_{t+m}
    \]
    each with an empty top cell, for some $t\in\{0,\ldots,n-m\}$.
\end{lemma}
\begin{proof}
    By definition, there are integers $s,h_0 \in \N$ such that $\CC$ is $(s,h_0)$-pattern-free.
    By \cref{thm:2wqo-dependent,thm:insulation-property}, $\CC$ is monadically dependent and has the insulation property.
    This means for the fixed height $h := h_0 + 1$ there are a function $N \from \N \to \N$ and an integer $q \in \N$ such that for every graph $G \in \CC$ and every set $X \subseteq V(G)$ of size at least $N(t)$, there is a subset $W \subseteq X$ of size at least $t$ that is $(h,q)$-insulated in $G$.
    Let
    \[
        B := 2q^{h_0} \cdot s + 2,
    \]
    be the bound from \cref{lem:bounded-height}, and define
    \[
        f(m) := N(B(m+1)).
    \]
    Consider $G \in \CC$ and $X \subseteq V(G)$ with $|X| \ge f(m)$.
    By the choice of $f$, there is a subset $W \subseteq X$ of size at least $B(m+1)$ and an insulator $\gc A = (A,\KK,F,R)$ of height $h$ and cost $q$ that insulates $W$.

    Let $I=(a_1,\ldots,a_n)$ be the indexing sequence of the grid $A$, listed in increasing order.
    Since $\gc A$ insulates $W$, every column $A[a_t,*]$ contains exactly one vertex of $W$ in its bottom cell.
    Hence $n=|W|\ge B(m+1)$.

    By \cref{lem:bounded-height}, fewer than $B$ columns of $\gc A$ have non-empty top cell.
    Therefore the set
    \[
        J:=\{t\in[n] : A[a_t,h]=\emptyset\}
    \]
    has size at least
    \[
        n-(B-1) \ge B(m+1)-(B-1) > Bm.
    \]
    The complementary set $[n]\setminus J$ has size at most $B-1$, so it splits the empty-top columns into at most $B$ intervals of consecutive indices.
    Since $|J|>Bm$, one of these intervals has length at least $m$.
    Hence, there is some $t\in\{0,\ldots,n-m\}$ such that
    \[
        A[a_{t+1},h]=\cdots=A[a_{t+m},h]=\emptyset.
    \]
    This proves the lemma.
\end{proof}

\subsection{Wrapping up}
\label{subsec:wrapup}
Armed with \Cref{lem:empty-top-insulator}, we finally prove the Double Separator Lemma.

\begin{samepage}
\begin{restatable}[Double Separator Lemma]{lemma}{lemPartitionDouble}\label{lem:partitionDouble}
    For every pattern-free graph class $\CC$, there are $f \from \N \to \N$ and $k \in \N$ such that for every graph $G \in \CC$ and set $X$ of size at least $f(m)$ there exists a size $m$ partition $P_1,\ldots,P_m$ of $V(G)$ such that:
    \begin{enumerate}
        \item $X$ intersects each of the sets $P_1,\ldots,P_m$,
        \item for all $1\leq \alpha < \beta \leq m$:
        \[
            \rk_G\left( M, L \uplus  R  \right) \leq k,
        \]
        where $L := P_{1} \cup \ldots \cup P_{\alpha-1}$,  $M := P_{\alpha+1} \cup \ldots \cup P_{\beta-1}$, and $R :=  P_{\beta+1} \cup \ldots \cup P_m$.
    \end{enumerate}
\end{restatable}
\end{samepage}

\begin{proof}
    Let $f$, $h$, $q$ be as given by \cref{lem:empty-top-insulator}. We show that $f$ and $k := 2^{qh}$ satisfy the conclusion.

    Consider $G$ and $X \subseteq V(G)$ as in the statement of the lemma.
    By \cref{lem:empty-top-insulator}, there is an insulator $\gc A = (A, \KK, F, R)$ of height $h$ and cost $q$ in $G$ that insulates a set $W \subseteq X$ and has the following property.
    \begin{property}
    There are consecutive column indices $1 \leq c_1 < c_2 < \ldots < c_m \leq n$, such that
    \[
        A[c_{1},h]=\cdots=A[c_{m},h]=\emptyset.
    \]
    (We assume, up to renaming, that $A$ is indexed by $(1,\ldots,n)$.)
    \end{property}
    
    \noindent
    We define a partition of $V(G)$ into $m$ parts by
    \begin{align*}
        P_1 &:= \bigl(V(G) \setminus A\bigr) \cup A[{\le}c_1, *], \\
        P_i &:= A[c_i, *] \quad \text{for } 2 \le i \le m-1, \\
        P_m &:= A[{\ge}c_m, *].
    \end{align*}
    By construction, each part contains an entire column of the insulator. As the insulator insulates $W \subseteq X$, this means each part intersects $X$, as desired.

    Fix $1 \le \alpha < \beta \le m$ and let $L, M, R$ be as in the statement.
    By the definition of the partition, $M = A[\{c_{\alpha+1},\ldots,c_{\beta-1}\},*]$.
    Since $c_{\alpha+1} > 1$ and $c_{\beta-1} < n$, the columns of $M$ are non-boundary.
    Since their top cells are empty, all vertices of $M$ lie in rows of height strictly less than $h$, so $M \subseteq \int(A)$.

    \begin{claim}\label{claim:homogeneous}
        For every vertex $v \in L \uplus R$, color $C \in \KK$, and row index $r \in [h]$, the vertex $v$ is homogeneous to $M \cap C \cap A[*,r]$ in $G$.
    \end{claim}
    \begin{claimproof}
        If $v \notin A$ is outside the grid, then the insulator property~\ref{itm:outside} says that $v$ is homogeneous to $C \cap \int(A)$ for each $C \in \KK$. Since $M \subseteq \int(A)$, the claim follows.

        If $v \in L \cap A$ is in the left side of the grid, then it is contained in some column $c_v \le c_\alpha - 1$ and row $s \in [h]$.
        Fix any two vertices $u, u' \in M \cap C \cap A[*,r]$.
        They share the same color $C$ and row $r < h$. Both lie in (possibly different) columns $c_u,c_{u'} \ge c_{\alpha}+1 \ge c_v + 2$ that are to the right of $c_v$ and non-consecutive with $c_v$.
        Applying property~\ref{itm:adjacency} to $u$ (and identically to $u'$):
        \begin{itemize}
            \item if $s \in \{r, r+1\}$, then  property~\ref{itm:adj-right} gives 
            \[
                uv \in E(G) \Leftrightarrow (\KK(v), C) \in R;    
            \]
            \item if $s = r-1$, then property~\ref{itm:adj-bot-left} gives $u,v$ non-adjacent in $G \oplus F$, hence 
            \[
                uv \in E(G) \Leftrightarrow (\KK(v), C) \in F;    
            \]
            \item if $|s-r| > 1$, then property~\ref{itm:adj-different-rows} (using $u \in \int(A)$) gives the same.
        \end{itemize}
        In each case $uv$ being an edge depends only on $C$, $\KK(v)$, and $s$; and the same holds for $u'v$. Thus,
        \[
            uv \in E(G)
            \Leftrightarrow
            u'v \in E(G)
        \]
        which proves the claim.

        If $v \in R \cap A$ is in the right side of the grid, then $v$ has column $c_v \ge  c_\beta + 1 \ge c_u + 2$ for any column $c_u$ of $M$.
        A symmetric argument using property~\ref{itm:adj-left} in place of~\ref{itm:adj-right} gives the same conclusion.
    \end{claimproof}

    By \cref{claim:homogeneous}, the neighborhood of any $v \in L \uplus R$ in $M$ is determined by the function
    $\KK \times [h] \to \{0,1\}$ recording for which $(C,r)$ the vertex $v$ is adjacent to $M \cap C \cap A[*,r]$.
    There are at most $2^{|\KK| \cdot h} \le 2^{qh} = k$ such functions.
    Since the rank of a $\{0,1\}$-matrix over the two element field is at most the number of distinct rows,
    $\rk_G(M, L \uplus R) \le k$.
    Since $1 \le \alpha < \beta \le m$ were arbitrary, the lemma follows.
\end{proof}

\end{document}